\documentclass[reqno, 12pt]{amsart}

\usepackage{amssymb}
\usepackage{fullpage}

\usepackage{xcolor}
\usepackage{hyperref}
\hypersetup{colorlinks=true,citecolor=purple,linkcolor=purple,urlcolor=gray,breaklinks=true}

\numberwithin{equation}{section}
\numberwithin{table}{section}
\newtheorem{thm}{Theorem}[section]
\newtheorem{prop}[thm]{Proposition}
\newtheorem{lem}[thm]{Lemma}
\newtheorem{cor}[thm]{Corollary}

\def\ldiv{\backslash}
\def\rdiv{/}
\newcommand{\m}{^{-1}}
\newcommand{\vhi}{\varphi}

\newcommand{\mlt}{\operatorname{Mlt}}
\newcommand{\fix}{\operatorname{Fix}}
\newcommand{\inn}{\operatorname{Inn}}
\newcommand{\atp}{\operatorname{Atp}}
\newcommand{\aut}{\operatorname{Aut}}
\newcommand{\id}{\operatorname{id}}

\newcommand{\nuc}{\operatorname{Nuc}}
\newcommand{\lps}{\operatorname{Psa}_\ell}

\newcommand{\lnuc}{\operatorname{Nuc}_\ell}

\newcommand{\indmod}[1]{\mathrm{mod}_{#1}}

\newcommand{\cref}[1]{Corollary~\ref{#1}}
\newcommand{\lref}[1]{Lemma~\ref{#1}}

\newcommand{\eref}[1]{(\ref{e#1})}

\newcommand{\pseu}{pseudo\-auto\-morphism}

\begin{document}
\title[Congruence solvability in Moufang loops]
{Congruence solvability in finite Moufang loops\\ of order coprime to three}
\author{Ale\v s Dr\'apal}
\address[Dr\'apal]{Dept.~of Mathematics\\ Charles University\\ Sokolovsk\'a 83\\
 186 75 Praha 8, Czech Republic}
\email[Dr\'apal]{drapal@karlin.mff.cuni.cz}

\author{Petr Vojt\v echovsk\'y}
\address[Vojt\v{e}chovsk\'y]{Dept.~of Mathematics\\ University of Denver\\ 2390
S.~York St.\\ Denver, CO 80208, USA}
\email[Vojt\v{e}chovsk\'y]{petr@math.du.edu}

\thanks{A.~Dr\'apal supported by the INTER-EXCELLENCE project LTAUSA19070 of M\v SMT Czech Republic. P.~Vojt\v{e}chovsk\'y supported by the Simons Foundation Mathematics and Physical Sciences Collaboration Grant for Mathematicians no.~855097 and by the PROF grant of the University of Denver.}



\begin{abstract}
We prove that a normal subloop $X$ of a Moufang loop $Q$ induces an abelian congruence of $Q$ if and only if each inner mapping of $Q$ restricts to an automorphism of $X$ and $u(xy) = (uy)x$ for all $x,y\in X$ and $u\in Q$. The former condition can be omitted when $X$ is $3$-divisible. This characterization is then used to show that classically solvable finite $3$-divisible Moufang loops are congruence solvable.
\end{abstract}

\maketitle

\section{Introduction}\label{i}

There are two notions of solvability in loop theory. \emph{Classical solvability} generalizes the standard definition of group theory, so a loop $Q$ is \emph{classically solvable} if there exists a series $Q=Q_0\ge Q_1\ge\cdots\ge Q_n=1$ such that $Q_{i+1}\unlhd Q_i$ and $Q_i/Q_{i+1}$ is a commutative group. Equivalently, a loop $Q$ is classically solvable it there exists a series $Q=Q_0\ge Q_1\ge\cdots\ge Q_n=1$ such that $Q_{i+1}\unlhd Q$ and $Q_i/Q_{i+1}$ is a commutative group. \emph{Congruence solvability} specializes a definition from congruence modular varieties, so a loop $Q$ is \emph{congruence solvable} if there exists a series $Q=Q_0\ge Q_1\ge\cdots\ge Q_n=1$ such that $Q_{i+1}\unlhd Q$ and the normal subloop $Q_i/Q_{i+1}$ of $Q/Q_{i+1}$ induces an abelian congruence of $Q/Q_{i+1}$.

Every congruence solvable loop is classically solvable but the converse does not hold in general. It is an open problem whether every classically solvable Moufang loop is congruence solvable. We proved in \cite{dravojt} that the two notions of solvability coincide in Moufang loops of odd order and in $6$-divisible Moufang loops.

A classically solvable nontrivial Moufang loop $Q$ contains a nontrivial commutative subgroup $X$ that is normal in $Q$. If $Q$ is also congruence solvable, $X$ can be assumed to satisfy additional properties which are easy to express (cf.~Theorem \ref{a2}) and which have structural implications for $Q$. Whether the two notions of solvability coincide in Moufang loops or not, congruence solvability has the potential to significantly influence and deepen the theory of Moufang loops.

\bigskip

In this paper we characterize normal subloops of $3$-divisible Moufang loops that induce an abelian congruence, and then we prove that the two notions of solvability coincide in finite $3$-divisible Moufang loops, improving upon the $6$-divisibility result of \cite{dravojt} in the finite case.

\bigskip

Here is a summary of the paper. The congruence commutator theory for loops was described in \cite{svcom}. In the follow-up paper \cite{svab}, Stanovsk\'y and Vojt\v echovsk\'y listed several equivalent conditions that characterize a normal subloop $X$ of $Q$ that induces an abelian congruence of $Q$, cf.~\cite[Theorem 4.1]{svab}. We record the relevant parts of their result as Theorem \ref{Th:svab}. Building upon Theorem \ref{Th:svab}, we show that in the case of Moufang loops, a normal subloop $X$ induces an abelian congruence of $Q$ if and only if (i) each inner mapping of $Q$ restricts to an automorphism of $X$, and (ii) $u(xy) =(uy)x$ for all $x,y\in X$ and $u\in Q$, cf.~Theorem \ref{a2}. We can further improve upon Theorem \ref{a2} in the case of $3$-divisible Moufang loops when condition (i) is automatically satisfied, cf.~Theorem \ref{a6}.

Given a $3$-divisible Moufang loop $Q$ and its normal subloop $S$, we then construct a certain subnormal series \eqref{Eq:Series} in the multiplication group $\mlt(Q)$ that starts with the relative multiplication group $\mlt_Q(S)$. One of the intermediate subloops of the series is constructed upon considering pseudoautomorphisms of $Q$.

In Section \ref{c} we prove that every classically solvable finite $3$-divisible Moufang loop is congruence solvable, cf.~Theorem \ref{Th:Equivalent}. The proof is an induction on the order of a smallest counterexample $Q$. As the three Isomorphism Theorems and the Correspondence Theorem hold in loops, the standard proof from group theory goes through to establish the following result: If $X$ is a normal subloop of a loop $Q$, then $Q$ is classically solvable if and only if both $X$ and $Q/X$ are classically solvable. Similarly, if $X$ is a normal subloop of $Q$ that induces an abelian congruence of $Q$, then $Q$ is congruence solvable if and only if both $X$ and $Q/X$ are congruence solvable.

It therefore suffices to find a nontrivial normal subloop $X$ of $Q$ that induces an abelian congruence of $Q$. If the nucleus $\nuc(Q)$ of $Q$ is nontrivial, it contains a nontrivial normal commutative subgroup $X$, and every such subgroup induces an abelian congruence of $Q$ by Proposition \ref{u4}.

If $\nuc(Q)$ is trivial then $\mlt(Q)$ admits certain triality automorphisms by \cite[Theorem 6]{GlaubermanII}. We study Moufang loops that admit triality automorphisms in the short Section \ref{triality}. Continuing the inductive proof, we then show that $Q$ contains a nontrivial normal $p$-subgroup $S$. The series \eqref{Eq:Series} yields  a subnormal series that starts with the $p$-group $\mlt_Q(S)$ (cf.~Theorem \ref{c1}) and terminates with $\mlt(Q)\rtimes S_3$, the multiplication group of $Q$ extended by the triality automorphisms. This gives rise to a nontrivial normal triality $p$-subgroup $U$ of $\mlt(Q)$ (cf.~Proposition \ref{c2}). Finally, $X=U(1)$ is then the sought-after normal subloop of $Q$ that induces an abelian congruence of $Q$ (cf.~Proposition \ref{u6}, whose proof uses the characterization of Theorem \ref{a6}).

The main results of this paper are Theorems \ref{a2}, \ref{a6} and \ref{Th:Equivalent}. However, many of the auxiliary results are likely to be useful in future investigations of congruence solvability in Moufang loops and in our understanding of the structure of Moufang loops.

\section{A characterization of abelian congruences in Moufang loops}\label{a}

See \cite{Bruck,Pflugfelder} for an introduction to loop theory and \cite{FM} for an introduction to commutator theory in congruence modular varieties.

Let $Q=(Q,\cdot,\ldiv,\rdiv,1)$ be a loop, a universal algebra satisfying the identities $x\ldiv (x\cdot y) = x\cdot (x\ldiv y) = y = (y\cdot x)/x = (y/x)\cdot x$ and $x\cdot 1 = x = 1\cdot x$ for all $x,y \in Q$. As usual, we also denote the multiplication operation $\cdot$ by juxtaposition and we assume that juxtaposition is more binding than the operations $\cdot$, $\ldiv$ and $\rdiv$. For instance, $xy \cdot (u\ldiv v) = (x\cdot y)\cdot (u\ldiv v)$.

A loop $Q$ is an \emph{inverse property loop} if for every $x\in Q$ there is $x\m\in Q$ such that $x\m\cdot xy = y = yx \cdot x\m$. This implies $1=xx\m = x\m x$, $(x\m)\m=x$, $x\ldiv y = x\m y$ and $x/y = xy\m$. Inverse property loops can therefore we treated as universal algebras in the signature $(\cdot, { }\m,1)$ familiar from groups.

An associative subloop of a loop $Q$ will be refereed to as a \emph{subgroup} of $Q$. A loop $Q$ is \emph{power associative} if any element of $Q$ generates a subgroup, and \emph{diassociative} if any two elements of $Q$ generate a subgroup. We will often not specify unnecessary parentheses in diassociative loops, e.g., we write $xyx$ instead of $x\cdot yx$ or $xy\cdot x$.

The loop identities
\begin{equation}\label{Eq:M}\tag{M}
    x (y\cdot xz) = (xy\cdot x)z,\ (zx \cdot y)x = z(x\cdot yx),\ x(yz\cdot x) =xy\cdot zx, (x\cdot yz) x = xy \cdot zx
\end{equation}
are pairwise equivalent and a loop satisfying any one (and hence all) of the identities is called a \emph{Moufang} loop. Every Moufang loop is diassociative \cite{Moufang}.

A power associative loop $Q$ is \emph{$d$-divisible} if the mapping $x\mapsto x^d$ is onto $Q$. We claim that a finite Moufang loop $Q$ is $3$-divisible if and only if the order of $Q$ is coprime to $3$. Certainly if $Q$ is not $3$-divisible then there are $x\ne y$ such that $x^3=y^3$, so the group $\langle x,y\rangle$ is not $3$-divisible, it contains an element of order $3$, and hence $3$ divides $|Q|$ by the elementwise Lagrange theorem for Moufang loops. Conversely, if $3$ divides $|Q|$ then $Q$ contains an element of order $3$ by \cite[Lemma 4]{Doro} and hence $Q$ is not $3$-divisible.

If $u$ is an element of a loop $Q$, then the \emph{left translation} and the \emph{right translations} by $u$ are defined by $L_u(v)=uv$ and $R_u(v)=vu$, respectively. The group $\mlt(Q) = \langle L_u,R_u:u\in Q\rangle$ is the \emph{multiplication group} of $Q$. An element of $\mlt(Q)$ that fixes $1$ is called an \emph{inner mapping}. The inner mappings of $Q$ form a subgroup $\inn(Q)$ of $\mlt(Q)$. It is well known that $\inn(Q) = \langle T_u,\,L_{u,v},\,R_{u,v}:u,v\in Q\rangle$, where
\begin{equation}\label{ea1}
    T_u = R_u^{-1}L_u,\ L_{u,v} = L_{uv}\m L_u L_v\text{ and }R_{u,v} = R_{uv}\m R_v R_u,
\end{equation}
are the \emph{standard generators} of $\inn(Q)$.

Let $\mathbf A$ be a universal algebra on an underlying set $A$. A congruence of $\mathbf A$ is an equivalence relation on $A$ that commutes with all operations of $\mathbf A$. As in \cite{FM}, a congruence $\alpha$ \emph{centralizes} congruence $\beta$ \emph{over} congruence $\gamma$ if for any $(n{+}1)$ary term operation $t$ in the signature of $\mathbf A$ the following implication holds for all $a,b,u_1,v_1,\dots,u_n,v_n\in A$ satisfying $a\,\alpha\,b$ and
$u_1\, \beta\, v_1$, $\dots$, $u_n\,\beta\,v_n$:
\begin{equation}\label{ea3}
t(a,u_1,\dots,u_n)\;\delta\;t(a,v_1,\dots,v_n)\ \Rightarrow \ t(b,u_1,\dots,u_n)\;\delta\;t(b,v_1,\dots,v_n)
\end{equation}
A congruence $\alpha$ is \emph{abelian} if the condition \eqref{ea3} holds with $\beta=\alpha$ and $\delta=\{(a,a):a\in A\}$.

There is a one-to-one correspondence between normal subloops of $Q$ and congruences of $Q$. For $X\unlhd Q$, the induced congruence $\indmod{X}$ is defined by $u\,\indmod{X}\,v$ iff $uX=vX$. For a congruence $\alpha$ of $Q$, the induced normal subloop is the equivalence class of $\alpha$ containing $1$.

Let $X\unlhd Q$. If $\indmod{X}$ is an abelian congruence of $Q$ then $X$ is a commutative group, but the converse is not true in general loops, not even in Moufang loops.

The following result is excerpted from \cite[Theorem 4.1]{svab}. It uses the somewhat unusual commutators and associators employed in \cite{svab}, namely, $[x,y] = ((xy)\rdiv x)\rdiv y$ and $[x,y,z] = ((xy\cdot z)/(yz))/x$. Note that we have $[x,y]=1$ iff $xy=yx$ and $[x,y,z]=1$ iff $xy\cdot z = x\cdot yz$.

\begin{thm}[\cite{svab}]\label{Th:svab}
Let $X$ be a normal subloop of a loop $Q$. Then the following conditions are equivalent:
\begin{enumerate}
\item[(i)] $X$ induces an abelian congruence of $Q$,
\item[(ii)] every inner mapping of $Q$ restricts to an automorphism of $X$ and for every $x,y\in X$ and $u,v,w\in Q$ with $v\,\indmod{X}\,w$ we have $[x,y] = [x,y,u] = [x,u,y] = [u,x,y]=1$ and $[x,u,v] = [x,u,w]$,
\item[(iii)] $Q$ is an \emph{abelian extension} of $X$ by $Q/X$, that is, $X$ is a commutative group, there is a transversal $T$ to $X$ in $Q$ containing $1$, for every $r,s\in T$ there are automorphisms $\varphi_{r,s},\psi_{r,s}\in\aut(X)$ and $\theta_{r,s}\in X$ such that $\varphi_{r,1} = \psi_{1,s}=\id_X$, $\theta_{1,s}=\theta_{r,1}=1$, and $rx\cdot sy = t\cdot \varphi_{r,s}(x)\psi_{r,s}(y)\theta_{r,s}$, where $t$ is the unique element of $T\cap (rs)X$.
\end{enumerate}
\end{thm}

We can substantially simplify condition (ii) in the case of Moufang loops:

\begin{thm}\label{a2}
Let $X$ be a normal subloop of a Moufang loop $Q$. Then $X$ induces an abelian congruence of $Q$ if and only if every inner mapping of $Q$ restricts to an automorphism of $X$ and $u\cdot xy = uy\cdot x$ for all $u \in Q$ and $x,y\in X$.
\end{thm}
\begin{proof}
The direct implication follows from Theorem \ref{Th:svab}(ii) as follows. Certainly every inner mapping of $Q$ restrict to an automorphism of $X$. Since $[x,y]=1$ for all $x,y\in X$, the normal subloop $X$ is commutative. Using commutativity and $[u,x,y]=1$ for $u\in Q$, $x,y\in X$, we deduce $u\cdot yx = u\cdot xy = ux\cdot y$.

For the converse implication, suppose that every inner mapping of $Q$ restricts to an automorphism of $X$ and $u\cdot xy = uy\cdot x$ for all $u \in Q$ and $x,y\in X$. The latter condition implies that $X$ is a commutative group. We will show that $Q$ is an abelian extension of $X$ by $Q/X$, as described in item (iii) of Theorem \ref{Th:svab}.

Let $T$ be a transversal to $X$ in $Q$ with $1\in T$. Let $r,s\in T$ and $x,y\in X$. In the following calculation we will use without reference the fact that every inner mapping of $Q$ restricts to an automorphism of $X$. We will also write $[Q,X,X]=1$ as a justification comment whenever we use $u\cdot xy = ux\cdot y$ with $u\in Q$ and $x,y\in X$. Finally, we omit parentheses in subterms involving only factors from the commutative group $X$. Now:
\begin{flalign*}
&& rx\cdot sy &= rx\cdot T_s(y)s = s(s^{-1}r\cdot L_{s^{-1},r}(x))\cdot T_s(y)s &&\\
&& &= s\cdot (s^{-1}r\cdot L_{s^{-1},r}(x))T_s(y)\cdot s &&\text{by \eqref{Eq:M}}\\
&& &= s\cdot (s^{-1}r\cdot L_{s^{-1},r}(x)T_s(y))\cdot s &&\text{$[Q,X,X]=1$}\\
&& &= s(s^{-1}r)\cdot L_{s^{-1},r}(x)T_s(y)s = r\cdot L_{s^{-1},r}(x)T_s(y)s &&\text{by \eqref{Eq:M}}\\
&& &= r\cdot sT_s^{-1}(L_{s^{-1},r}(x)T_s(y)) = rs\cdot L_{r,s}T_s^{-1}(L_{s^{-1},r}(x)T_s(y)) &&\\
&& &= tz\cdot L_{r,s}T_s^{-1}(L_{s^{-1},r}(x)T_s(y)) &&\text{$t\in T\cap (rs)X$, $z\in X$}\\
&& &= t\cdot zL_{r,s}T_s^{-1}(L_{s^{-1},r}(x)T_s(y)) &&\text{$[Q,X,X]=1$}\\
&& & = t\cdot L_{r,s}T_s^{-1}(L_{s^{-1},r}(x)T_s(y))z &&\text{$X$ is commutative}\\
&& & = t\cdot ( L_{r,s}T_s^{-1}L_{s^{-1},r}(x)\cdot L_{r,s}T_s^{-1}T_s(y)\cdot z). &&
\end{flalign*}
Hence $rx\cdot sy = t\cdot \vhi_{r,s}(x)\psi_{r,s}(y)\theta_{r,s}$, where $\vhi_{r,s}$ is the restriction of $L_{r,s}T_s^{-1}L_{s^{-1},r}$ to $X$, $\psi_{r,s}$ is the restriction of $L_{r,s}T_s^{-1}T_s = L_{r,s}$ to $X$ and $\theta_{r,s}=z\in X$. If $r=1$ then $t$ is the unique element of $T\cap sX$, so $t=s$ and $\theta_{1,s}=z=1$. Similarly, $\theta_{r,1}=1$. We also have $\vhi_{r,1} = L_{r,1}T_1^{-1}L_{1,r}=\id_X$ and $\psi_{1,s} = L_{1,s}=\id_X$.
\end{proof}

We proceed to improve upon Theorem \ref{a2} in the case of $3$-divisible Moufang loops. We start by recalling two results of Gagola.

\begin{prop}[dual of {\cite[Theorem 1]{gagcyc}}]\label{Pr:G1}
Let $Q$ be a Moufang loop. Then
\begin{displaymath}
    u^{3i}x\cdot u^{3j}y =  u^{3(i+j)}T_u^{-i-2j}(T_u^{i-j}(x)T_u^{i-j}(y))
\end{displaymath}
for all $u,x,y\in Q$ and all $i,j\in\mathbb Z$.
\end{prop}

We will only need Proposition \ref{Pr:G1} for the case $i=1$ and $j=0$. See \cite{dravojt} for a quick proof of that case.

\begin{prop}[{\cite[Theorem 1]{gag3}}]\label{Pr:G2}
Let $Q$ be a Moufang loop generated by a set of elements, each of which is a cube of an element of $Q$. Then $\inn(Q) = \langle T_u:u\in Q\rangle$.
\end{prop}

\begin{thm}\label{a6}
Let $X$ be a normal subloop of a $3$-divisible Moufang loop $Q$. Then $X$ induces an abelian congruence of $Q$ if and only if $u\cdot xy =uy\cdot x$ for all $u\in Q$ and $x,y\in X$.
\end{thm}
\begin{proof}
The direct implication follows from Theorem \ref{a2}. For the converse implication, suppose that $u\cdot xy =uy\cdot x$ for all $u\in Q$ and $x,y\in X$ so that $X$ is a commutative group. Since $Q$ is $3$-divisible, every element of $Q$ is a cube of some element of $Q$ and Proposition \ref{Pr:G2} therefore applies. In view of Proposition \ref{Pr:G2} and Theorem \ref{a2}, it suffices to show that $T_u$ restricts to an automorphism of $X$ for every $u\in Q$. Let $x,y\in X$. With $i=1$ and $j=0$, the formula of Proposition \ref{Pr:G1} becomes $u^3x\cdot y = u^3T_u^{-1}(T_u(x)T_u(y))$. Since $u^3x\cdot y = u^3\cdot yx = u^3\cdot xy$ is assumed, we have $u^3\cdot xy = u^3T_u^{-1}(T_u(x)T_u(y))$, $xy = T_u^{-1}(T_u(x)T_u(y))$ and $T_u(xy) = T_u(x)T_u(y)$.
\end{proof}

\section{A subnormal series in the multiplication group}\label{d}

If $S\le Q$ and $Q$ is a loop then $\mlt_Q(S)= \langle L_s,R_s:s\in S\rangle\le\mlt(Q)$ is the \emph{relative multiplication group} of $S$ in $Q$.

If $S\unlhd Q$ then $\mlt(Q/S)$ is an image of $\mlt(Q)$ under the homomorphism determined by $L_x\mapsto L_{xS}$ and $R_x\mapsto R_{xS}$. Let $\mathcal C(Q,S)$ be the kernel of this homomorphism. It is not hard to show that $\mathcal C(Q,S)$ consists of all $\vhi\in\mlt(Q)$ that \emph{centralize the cosets} of $S$ in $Q$, that is, $\varphi(uS) = uS$ for all $u\in Q$.

\begin{lem}\label{Lm:Decomp}
Let $S$ be a normal subloop of $Q$. Then
\begin{displaymath}
	\mathcal C(Q,S) = \{L_s\sigma,R_s\sigma:s\in S,\,\sigma\in\inn(Q)\cap\mathcal C(Q,S)\}.
\end{displaymath}
\end{lem}
\begin{proof}
Given $\vhi\in\mlt(Q)$, there are uniquely determined $v\in Q$ and $\sigma\in\inn(Q)$ such that $\vhi=L_v\sigma$. If $\sigma\in\inn(Q)\cap\mathcal C(Q,S)$ and $s\in S$ then $L_s\sigma(uS) = L_s(uS) = L_s(Su) = s(Su) = (sS)u = Su = uS$ for all $u\in Q$, and therefore $\sigma\in\inn(Q)\cap\mathcal C(Q,S)$ and $L_s\sigma\in\mathcal C(Q,S)$. Conversely, let $\vhi=L_v\sigma\in\mathcal C(Q,S)$. Since $S=\vhi(S)= L_v\sigma(S) = L_v(S) = vS$, it follows that $v\in S$. Moreover, $\sigma(uS) = L_v^{-1}L_v\sigma(uS) = L_v^{-1}(uS) = uS$ since $v\in S$. The argument for $R_s\sigma$ is similar.
\end{proof}

In particular, $\mlt_Q(S)=\langle L_s,R_s:s\in S\rangle \le\mathcal C(Q,S)\unlhd \mlt(Q)$. In this section we will refine this series into a subnormal series in the case of $3$-divisible Moufang loops.

For a loop $Q$, define the \emph{nucleus} $\nuc(Q)$ and the \emph{left nucleus} $\lnuc(Q)$ by
\begin{align*}
    \nuc(Q)&=\{a\in Q:a(uv)=(au)v,\,u(av)=(ua)v,\,u(va)=(uv)a\text{ for all }u,v\in Q\},\\
    \lnuc(Q) &= \{a\in Q:a(uv)=(au)v\text{ for all }u,v\in Q\}.
\end{align*}
Bruck proved in \cite{Bruck} that the nuclei are subgroups of $Q$, and if $Q$ is Moufang then $\nuc(Q)=\lnuc(Q)$ is a normal subloop of $Q$.

A (\emph{left}) \emph{\pseu} $\vhi$ of $Q$ is a permutation of $Q$ for which there exists some $c\in Q$ such that
\begin{displaymath}
	c\vhi(x) \cdot \vhi(y) = c\vhi(xy)
\end{displaymath}
for all $x,y \in Q$. The element $c$ is called a (\emph{left}) \emph{companion} of $\vhi$. Then $d\in Q$ is another companion of $\vhi$ if and only if $d/c\in\lnuc(Q)$.

Denote by $\lps(Q)$ the set of all pairs $(c,\vhi)$ such that $c$ is a left companion of a pseudoautomorphism $\vhi$ of $Q$. Then $\lps(Q)$ is a group with operations
\begin{equation}\label{ed1}
    (c,\vhi)(d,\psi) = (c\vhi(d),\vhi\psi) \text{\, and \,}
    (c,\vhi)\m = (\vhi\m(c\m),\vhi\m).
\end{equation}

Pseudoautomorphisms may be interpreted by means of autotopisms. A triple $(\alpha,\beta,\gamma)$ of permutations of $Q$ is an \emph{autotopism} of $Q$ if
\begin{equation}\label{Eq:Atp}
    \alpha(x)\beta(y) = \gamma(xy)
\end{equation}
for all $x,y \in Q$. Clearly,
\begin{equation}\label{ed2}
    (c,\vhi) \in \lps(Q)\quad \Leftrightarrow \quad (L_c\vhi,\vhi,L_c\vhi) \in \atp(Q).
\end{equation}
The set of all autotopisms of $Q$ forms the \emph{autotopism group} $\atp(Q)$ under componentwise composition. If $(\alpha,\beta,\gamma)\in \atp(Q)$ and $\beta(1) =1$, then $\alpha = \gamma$ and $(\alpha(1),\beta) = (\gamma(1),\beta)\in \lps(Q)$. This can be easily seen by setting first $y=1$ and then $x=1$ in \eqref{Eq:Atp}.

Let us now summarize basic facts pertaining to autotopisms and \pseu s in Moufang loops, cf.~\cite{Bruck}.

Let $Q$ be a Moufang loop. Every inner mapping of $Q$ is a \pseu{}. This can be seen by verifying the well-known identities
\begin{equation}\label{ed4}
    \begin{gathered}
    \text{$(x^{-3},T_x)\in \lps(Q)$, $([x\m,y],[L_x,R_y])\in \lps(Q)$,}\\
\text{$L_{x,y} = [L_x,R_y\m]=[R_x\m,L_y]\,$ and $R_{x,y} = [L_y\m,R_x] =
[R_y,L_x\m]$.}
\end{gathered}
\end{equation}

Let $M_x = L_xR_x$. Since $Q$ is Moufang, $M_x$ is also equal to $R_xL_x$. If $x\in Q$ then $\atp(Q)$ contains each of the triples
\begin{equation}\label{ed3}
(L_x,R_x,M_x),\ (M_x,L_x\m,L_x) \text{\, and } (R_x\m, M_x, R_x).
\end{equation}

A \emph{semiautomorphism} of a Moufang loop $Q$ is a bijection $\vhi$ of $Q$ satisfying $\vhi(1)=1$ and $\vhi(xyx)=\vhi(x)\vhi(y)\vhi(x)$ for all $x,y\in Q$. Every semiautomorphism satisfies $\vhi(x^i)=\vhi(x)^i$ for all $x\in Q$ and $i\in\mathbb Z$. Every pseudoautomorphism of a Moufang loop is a semiautomorphism. In particular, $\vhi(x^{-1}) = x^{-1}$ if $\vhi\in\inn(Q)$.

\begin{lem}\label{d1}
Let $Q$ be a Moufang loop. Suppose that $x\in Q$ and $(c,\vhi)\in \lps(Q)$. Then
\begin{displaymath}
    (c^{\vhi(x\m)} ,L_{\vhi(x)}\m \vhi L_x)\in \lps(Q),
\end{displaymath}
where $c^y = y^{-1}cy$.
\end{lem}
\begin{proof}
Let $\psi = L_{\vhi(x)}\m \vhi L_x$ and note that $\psi(1) = 1$. By \eqref{ed2} and \eref{d3} the composition
\begin{displaymath}
    (M_{\vhi(x)},L_{\vhi(x)}\m,L_{\vhi(x)})(L_c\vhi,\vhi,L_c\vhi)(M_x\m,L_x,L_x\m)
\end{displaymath}
is an autotopisms of $Q$. Hence a companion of $\psi = L_{\vhi(x)}\m\vhi L_x$ is equal to $L_{\vhi(x)}L_c\vhi L_x\m(1) = \vhi(x)c\vhi(x\m) = \vhi(x^{-1})^{-1}c\vhi(x^{-1})$.
\end{proof}

The following result will be useful in the inductive proof of Theorem \ref{Th:Equivalent}.

\begin{lem}\label{Lm:ElAb}
Let $A$ be a finite commutative normal subgroup of a Moufang loop $Q$. Let $p$ be a prime dividing $|A|$. Then there is a nontrivial normal $p$-subgroup of $A$ that is normal in $Q$.
\end{lem}
\begin{proof}
Let $S$ be the $p$-primary component of the commutative group $A$. Let $\vhi$ be an inner mapping of $Q$. Since $A\unlhd Q$, $\vhi(A)=A$. As $\vhi$ is a semiautomorphism of $Q$, it restricts to a semiautomorphism of $A$. Recall that $\vhi(x^i)=\vhi(x)^i$ for all $x\in Q$ and $i\in\mathbb Z$. In particular, if $x\in S$ then $|\vhi(x)|$ divides $|x|$ and therefore $\vhi(x)\in S$.
\end{proof}

Let $S$ be a normal subloop of a Moufang loop $Q$. Recall the normal subgroup $\mathcal C(Q,S)$ of $\mlt(Q)$ and define $\mathcal C_0(Q,S)$ as the set of all $\vhi\in\mathcal C(Q,S)$ such that when $\vhi$ is written as $\vhi=L_s\sigma$ with $s\in S$ and $\sigma\in\inn(Q)$ (cf.~Lemma \ref{Lm:Decomp}) then the  pseudoautomorphism $\sigma$ has a companion in $S$.

\begin{prop}\label{d2}
Let $S$ be a normal subloop of a Moufang loop $Q$. Then $\mathcal C_0(Q,S)$ is a normal subgroup of $\mathcal C(Q,S)$.
\end{prop}
\begin{proof}
Let $\vhi = L_s\sigma\in\mathcal C(Q,S)$ with $s\in S$ and $\sigma\in\inn(Q)$, and let $c\in Q$ be a companion of $\sigma$. We show that the mapping
\begin{displaymath}
    f:\mathcal C(Q,S)\to Q/(\nuc(Q)S),\quad L_s\sigma\mapsto c (\nuc(Q)S)
\end{displaymath}
is a well-defined homomorphism with kernel $\mathcal C_0(Q,S)$.

Since since both $S$ and $\nuc(Q)$ are normal in $Q$, they generate the normal subloop $\nuc(Q)S\unlhd Q$. If $c$ and $d$ are companions of $\sigma$ then $d\in c\nuc(Q)\subseteq c\nuc(Q)S$, so $f$ is well-defined. For the homomorphic property, consider $L_s\sigma$, $L_t\tau\in\mathcal C(Q,S)$ (with $s,t\in S$ and $\sigma,\tau\in\inn(Q)$) such that $(c,\sigma)$, $(d,\tau)\in\lps(Q)$ for some $c,d\in Q$. We have $f(L_s\sigma)f(L_t\tau) = c\nuc(Q)S\cdot d\nuc(Q)S = (cd)\nuc(Q)S$. Let
\begin{displaymath}
    \psi = L_{s\sigma(t)}\m L_s\sigma L_t\tau = L_{s\sigma(t)}\m L_s L_{\sigma(t)} L_{\sigma(t)}\m \sigma L_t \tau
\end{displaymath}
and observe that $\psi(1) = 1$, so $\psi\in\inn(Q)$. To compute a companion of $\psi$, note that $[s\m,\sigma(t\m)]$ is a companion of $L_{s\sigma(t)}\m L_s L_{\sigma(t)}=L_{s,\sigma(t)} = [L_s,R_{\sigma(t)}\m]$, by \eref{d4}, and that $c^{\sigma(t\m)}$ is a companion of $L_{\sigma(t)}\m \sigma L_t$, by \lref{d1}. Using \eref{d1}, $L_{\sigma(t)}^{-1}\sigma L_t\tau$ has a companion $c^{\sigma(t^{-1})}\cdot L_{\sigma(t)}^{-1}\sigma L_t(d) = c^{\sigma(t^{-1})}\cdot \sigma(t\m)\sigma(td)$. Using \eref{d1} again shows that $\psi$ possesses a companion
\begin{displaymath}
    e=[s\m,\sigma(t\m)]\,[L_s,R_{\sigma(t)}\m]\left(c^{\sigma(t\m)}\cdot \sigma(t\m)\sigma(td)\right).
\end{displaymath}
By Lemma \ref{Lm:Decomp}, each of $\sigma$, $\tau$, $L_{s\sigma(t)}$, $L_{\sigma(t)}$ and $L_s$ belong to $\mathcal C(Q,S)$. Hence $\psi\in\inn(Q)\cap\mathcal C(Q,S)$. Since $s\sigma(t)\in S$, we see that $L_s\sigma L_t\tau$ decomposes as $L_{s\sigma(t)}\psi$ and $f(\psi) = e\nuc(Q)S$. For the homomorphic property, it remains to show that $e\equiv cd$ modulo $\nuc(Q)S$. Since $\sigma$ centralizes cosets of $S$ and we work modulo $\nuc(Q)S\ge S$, $e$ is equivalent to
\begin{displaymath}
	[s^{-1},t^{-1}][L_s,R_{t}\m](c^{(t\m)}\cdot t\m(td)).
\end{displaymath}
Since $s,t\in S$, we have further $e\equiv [L_s,R_t\m](cd)$. The left translation $L_s$ is identical modulo $S$, and $e\equiv cd$ follows.

The kernel of $f$ consist of all $L_s\sigma\in \mathcal C(Q,S)$ such that $\sigma$ has a companion in $\nuc(Q)S$. Since the companions are determined up to $\nuc(Q)$, the kernel coincides with $\mathcal C_0(Q,S)$.
\end{proof}

\begin{prop}\label{d3}
Let $S$ be a normal subloop of a Moufang loop $Q$. Is $S$ is $3$-divisible then $\mlt_Q(S)\unlhd \mathcal C_0(Q,S)$.
\end{prop}
\begin{proof}
Let $t\in S$ and $L_s\sigma\in\mathcal C_0(Q,S)$, where $s\in S$ and $\sigma\in\inn(Q)\cap\mathcal C(Q,S)$ has a companion $c\in S$. We need to show that  $L_t^{L_s\sigma}, R_t^{L_s\sigma}\in\mlt_Q(S)$. Since $L_t^{L_s\sigma} = (L_s^{-1}L_tL_s)^\sigma = (L_s^{-1})^\sigma L_t^\sigma L_s^\sigma$, $(L_s^{-1})^\sigma = (L_s^\sigma)^{-1}$ and similarly for $R_t^{L_s\sigma}$, we only need to show that $L_s^\sigma,R_s^\sigma\in\mlt_Q(S)$. We will prove $L_s^\sigma\in\mlt_Q(S)$, the other case following dually.

Since $S$ is 3-divisible, there is $d\in S$ such that $d^3 =c$. Note that $T_d(c)=c$. By \eqref{ed1}, $\lps(Q)$ contains $(c\m,T_d)(c,\sigma) = (c\m T_d(c),T_d\sigma) = (1,T_d\sigma)$, which means that $\alpha = T_d\sigma$ is an automorphism of $Q$. Recall that for all $x\in Q$ we have $L_x^\alpha = \alpha\m L_x\alpha = L_{\alpha{-1}(x)}$, $R_x^\alpha = R_{\alpha^{-1}(x)}$ and thus $T_x^\alpha = T_{\alpha^{-1}(x)}$. Also, $T_x^{-1} = T_{x^{-1}}$ in a Moufang loop. Therefore
\begin{displaymath}
    L_s^\sigma = (L_s^{T_d\m})^\alpha = (T_{d}L_sT_{d\m})^\alpha =T_{\alpha\m(d)}L_{\alpha\m(s)}T_{\alpha\m(d\m)}
\end{displaymath}
is in $\mlt_Q(S)$, since $\alpha(S)=T_d\sigma(S)=T_d(S)=S$.
\end{proof}

\begin{cor}\label{Cr:Series}
Let $S$ be a normal $3$-divisible subloop of a Moufang loop $Q$. Then
\begin{equation}\label{Eq:Series}
    \mlt_Q(S)\unlhd \mathcal C_0(Q,S)\unlhd C(Q,S)\unlhd\mlt(Q).
\end{equation}
\end{cor}

\section{Moufang loops admitting triality automorphisms}\label{triality}

Glauberman observed in \cite[Theorem 6]{GlaubermanII} that if $Q$ is a Moufang loop with trivial nucleus, then the mappings
\begin{align*}
    &\sigma: L_x\mapsto R_x^{-1},\,R_x\mapsto L_x^{-1}\\
    &\rho: L_x\mapsto R_x,\,R_x\mapsto M_x^{-1}=L_x^{-1}R_x^{-1}\ ( \text{and }M_x\mapsto L_x^{-1})
\end{align*}
extend uniquely to automorphisms of $\mlt(Q)$. See \cite{Doro,Phillips} and \cite[Chapter 13]{Hall} for more information on groups with triality and Moufang loops.

We say that a Moufang loop $Q$ \emph{admits triality automorphisms} if the above maps $\sigma$ and $\rho$ extend into automorphisms of $\mlt(Q)$. (There exist Moufang loops with nontrivial nucleus that admit triality automorphisms.)

Suppose that a Moufang loop $Q$ admits triality automorphisms. Since the subgroup $\langle\sigma,\rho\rangle$ satisfies the standard presenting relations of the symmetric group $S_3$, it is isomorphic to a homomorphic image of $S_3$. Hence $Q$ induces a semidirect product $\mlt(Q)\rtimes S_3$ (with an action of $S_3$ that might not be faithful). Let $\alpha = \sigma\rho$ and $\beta = \sigma\rho^2$, so that $\alpha(L_x) = L_x^{-1}$, $\alpha(R_x) = M_x$, $\beta(L_x)=M_x$,  $\beta(R_x)=R_x^{-1}$ and $\sigma = \alpha\beta\alpha$. Note that $\sigma$ centralizes $\inn(Q)$. Indeed, $\sigma(T_x) = \sigma(R_x^{-1}L_x) = L_xR_x^{-1}=T_x$, $\sigma(L_{x,y}) = \sigma([L_x,R_y^{-1}]) = [R_x^{-1},L_y] = L_{x,y}$ and $\sigma(R_{x,y}) = \sigma([L_y^{-1},R_x]) = [R_y,L_x^{-1}] = R_{x,y}$, where we have used \eqref{ed4}.

Let $Q$ be a Moufang loop that admits triality automorphisms. A subgroup $U\le\mlt(Q)$ is a \emph{triality subgroup} if $U$ is invariant under the triality automorphisms $\sigma$ and $\rho$. If $U\unlhd \mlt(Q)$ then $U$ is a triality subgroup if and only if $U \unlhd \mlt(Q) \rtimes S_3$ under the induced action of $S_3$.

\begin{lem}\label{u3}
Suppose that $Q$ is a Moufang loop that admits triality automorphisms and that $U\unlhd \mlt(Q)$ is a triality subgroup. Let $S = U(1)$ be the orbit of $U$ that contains the element $1$. Then $S$ is a normal subloop of $Q$ and $\mlt_Q(S) \le U$.
\end{lem}
\begin{proof}
Since $U$ is normal, the blocks conjugate to $S$ in $\mlt(Q)$ form equivalence classes of a congruence of $Q$. This implies that $S$ is normal in $Q$. For each $s\in S$ there is $\vhi \in \inn(Q)$ such that $L_s\vhi \in U$. Then $M_s = L_sR_s  = L_s\vhi(R_s^{-1}\vhi)^{-1} = L_s\vhi\sigma(L_s\vhi)^{-1}\in U$. Hence also $R_s = \alpha(M_s)\in U$ and $L_s = \beta(M_s)\in U$.
\end{proof}

\begin{prop}\label{u6}
Let $Q$ be a $3$-divisible Moufang loop that admits triality automorphisms. Let $U\unlhd \mlt(Q)$ be a nontrivial commutative triality subgroup. Then $Q$ possesses a nontrivial subloop $X\unlhd Q$ that induces an abelian congruence of $Q$.
\end{prop}
\begin{proof}
By Lemma \ref{u3}, $X = U(1)$ is a normal subloop of $Q$ such $L_x,R_x\in U$ for all $x\in X$. Since $U$ is commutative, $[L_x,L_y] = \id_Q = [L_x,R_y]$ for all $x,y\in Q$. The first condition implies that $X$ is commutative. The second condition says that $x\cdot uy = xu\cdot y$ for all $x,y\in X$, $u\in Q$. Then, by Moufang Theorem \cite{Moufang, DrapalMouf}, $u\cdot xy = ux\cdot y$ for all $x,y\in X$, $u\in Q$.  Hence $u\cdot yx = u\cdot xy = ux\cdot y$ for all $x,y\in X$, $u\in Q$. By Theorem \ref{a6}, $X$ induces an abelian congruence of $Q$.
\end{proof}

\section{Solvability in finite Moufang loops of order coprime to three}\label{c}

In this section we show that the two concepts of solvability coincide for finite $3$-divisible Moufang loops. We start with a well known fact of elementary group theory:

\begin{lem}\label{d4}
Let $H$ be a subnormal subgroup of a finite group $G$. If there exists a nontrivial $p$-group $N\unlhd H$, $p$ a prime, then $G$ contains a nontrivial normal $p$-subgroup.
\end{lem}
\begin{proof}
Without loss of generality it may be assumed that $N$ is a minimal normal $p$-subgroup of $H$. Thus there exists $U\le H$ such that $N\le U$ and $U$ is a minimal characteristic subgroup of $H$. Since $N\unlhd U$, the group $U$ has to be a $p$-group.

If $H\unlhd G$ then $U\unlhd G$ since $U$ is characteristic in $H$. This proves the case $k=1$ of the general case $N\unlhd H = H_k \unlhd \cdots \unlhd H_1 \unlhd H_0=G$. Suppose that $k>1$ and proceed by induction. By the induction assumption there exists a nontrivial normal $p$-subgroup of $H_1$. We are done by the case $k=1$.
\end{proof}

Next, let us recall a result of Glauberman and Wright on Moufang loops of prime power order. The odd case was established in \cite{GlaubermanII} and the even case in \cite{GW}. A new proof that covers both cases can be found in \cite{dranil}.

\begin{thm}[\cite{GlaubermanII,GW}]\label{Th:GGW}
Let $Q$ be a Moufang loop of prime power order $p^k$. Then $Q$ is centrally nilpotent and $\mlt(Q)$ is a $p$-group.
\end{thm}

We will not use Theorem \ref{Th:Gpi} below but we offer it as a motivation for Theorem \ref{c1}, which is taken from \cite{dranil}. We have included a proof of Theorem \ref{c1} for the convenience of the reader.

For a set of primes $\pi$, a power associative loop $Q$ is a \emph{$\pi$-loop} if for every $x\in Q$ the order of $x$ is a power of a prime from $\pi$.

\begin{thm}[{\cite[Theorem 3]{GlaubermanII}}]\label{Th:Gpi}
Let $Q$ be a Moufang loop of odd order, $\pi$ a set of primes and $S$ a classically solvable $\pi$-subloop of $Q$. Then $\mlt_Q(S)$ is a solvable $\pi$-group.
\end{thm}

\begin{thm}[\cite{dranil}]\label{c1}
Let $Q$ be a finite Moufang loop, $p$ a prime and $S$ a $p$-subloop of $Q$. Then $\mlt_Q(S)$ is a $p$-group.
\end{thm}
\begin{proof}
Restricting $\vhi\in\mlt_Q(S)$ to $\mlt(S)$ yields an epimorphism with kernel $\fix_Q(S) = \{\vhi\in\mlt_Q(S):\vhi(s)=s$ for all $s\in S\}$. By Theorem \ref{Th:GGW}, $\mlt(S)$ is a $p$-group and therefore $\mlt_Q(S)/\fix_Q(S)$ is a $p$-group.

Consider the group $\inn_Q(S)=\mlt_Q(S)\cap\inn(Q) = \langle T_s,L_{s,t},R_{s,t}:s,t\in S\rangle$. For each $\vhi\in\inn_Q(S)$, let $C(\vhi)$ be the set of all companions of $\vhi$, a coset of $N=\nuc(Q)$. Since the standard generators of $\inn_Q(S)$ are pseudoautomorphisms with companions in $S$ and every $\vhi\in\inn_Q(S)$ satisfies $\vhi(S)=S$, it follows from \eqref{ed1} that $C(\vhi)\cap S\ne\emptyset$.

For $\vhi,\psi \in \fix_Q(S)$, we therefore certainly have $c,d\in S$ such that  $(c,\vhi)$, $(d,\psi)\in \lps(Q)$. Since $\vhi(d) = d$, we also have $(cd,\vhi\psi) = (c,\vhi)(d,\psi)\in \lps(Q)$ by \eqref{ed1}, proving that $cd\in C(\vhi\psi)$. The element $cd$ also belongs to $C(\vhi)C(\psi) = cNdN = cdN$. Hence $\vhi\mapsto C(\vhi)$ is
a homomorphism from the group $\fix_Q(S)$ to the loop $SN/N \cong S/(S\cap N)$. Let $A$ be the kernel of this homomorphism. Being associative, the image of the homomorphism is equal to a subgroup of $S/(S\cap N)$, some $p$-group. Hence $\fix_Q(S)/A$ is a $p$-group.

The kernel $A$ consists of all automorphisms of $Q$ that are contained in $\fix_Q(S)$. Now, $\alpha L_s\alpha^{-1} = L_{\alpha(s)} = L_s$ for every $s\in S$ and $\alpha \in A$. Similarly for $R_s$. Hence $A \le Z(\mlt_Q(S))$ is a commutative group. Write $A$ as $B\times D$, where $B$ is the $p$-primary component of $A$. Both $B$ and $D$ are normal subgroups of $\mlt_Q(S)$, being central. All three $\mlt_Q(S)/\fix_Q(S)$, $\fix_Q(S)/A$ and $A/D$ are $p$-groups. Hence $\mlt_Q(S)/D$ is a $p$-group. The subgroup $D$ is commutative, normal and of order coprime to $\mlt_Q(S)/D$. Hence it possesses a complement in $\mlt_Q(S)$, say $P$, a Sylow $p$-subgroup of $\mlt_Q(S)$. Since $D\le Z(\mlt_Q(S))$, $P$ is a normal subgroup of $\mlt_Q(S)$ and hence the unique Sylow $p$-subgroup of $\mlt_Q(S)$. For $s\in S$, we have $|L_s|=|R_s|=|s|$ by diassociativity. Since $S$ is a $p$-loop and the elementwise Lagrange theorem holds in Moufang loops, it follows that both $L_s$ and $R_s$ belong to $P$. Hence $P=\mlt_Q(S)=\langle L_s,R_s:s\in S\rangle$.
\end{proof}

\begin{prop}\label{u4}
Let $X$ be a commutative normal subloop of a Moufang loop $Q$. If $X\le\nuc(Q)$ then $X$ induces an abelian congruence of $Q$.
\end{prop}
\begin{proof}
By \cite{Bruck}, $\nuc(Q) \unlhd Q$. Hence if $a,b\in\nuc(Q)$ and $x\in Q$ then
\begin{equation}\label{eu4}
    T_x(ab) = (x\cdot ab)x^{-1} = (xa\cdot b)x\m = (T_x(a)x\cdot b)x\m = (T_x(a)\cdot xb)x\m = T_x(a)T_x(b),
\end{equation}
where we used $T_x(a)\in\nuc(Q)$ in the last step. The remaining standard generators of $\inn(Q)$ act trivially on $\nuc(Q)$. By Theorem \ref{a2}, $X$ induces an abelian congruence of $Q$.
\end{proof}

\begin{prop}\label{c2}
Let $Q$ be a finite $3$-divisible Moufang loop with a nontrivial normal $p$-subloop $S$, $p\ne 3$ a prime. Then $\mlt(Q)$ contains a nontrivial normal elementary abelian $p$-subgroup. Furthermore, if $Q$ also admits triality automorphisms, then $\mlt(Q)$ contains a nontrivial normal elementary abelian triality $p$-subgroup.
\end{prop}
\begin{proof}
It suffices to prove that $\mlt(Q)$ contains a nontrivial normal $p$-subgroup since the center of such a $p$-group is characteristic, and the socle of an abelian $p$-group is characteristic too. By \lref{d4}, it even suffices to show the existence of a subnormal $p$-group.

Corollary \ref{Cr:Series} implies the existence of the subnormal series \eqref{Eq:Series}. By Theorem \ref{c1}, $\mlt_Q(S)$ is a $p$-group.

If $Q$ admits triality automorphisms, the subnormal series \eqref{Eq:Series} may be extended by $\mlt(Q)\unlhd \mlt(Q)\rtimes S_3$.
\end{proof}

\begin{thm}\label{Th:Equivalent}
Let $Q$ be a finite $3$-divisible Moufang loop. Then $Q$ is classically solvable if and only if it is congruence solvable.
\end{thm}
\begin{proof}
The converse implication holds in general. For the direct implication, let $Q$ be a smallest $3$-divisible Moufang loop that is classically solvable but not congruence solvable. If there is a normal subloop $X$ of $Q$ that induces an abelian congruence of $Q$, then $Q/X$ is $3$-divisible and classically solvable, hence congruence solvable by minimality of $Q$, but then $Q$ is congruence solvable, being an abelian extension of a commutative group $X$ by a congruence solvable loop $Q/X$, a contradiction.

Suppose that $1<N=\nuc(Q)\unlhd Q$. Since $Q$ is classically solvable, the group $N$ is solvable. It therefore contains a nontrivial characteristic commutative subgroup $X\unlhd N$. By \eref{u4}, $\inn(Q)$ acts upon $X$ as a subgroup of $\aut(N)$. Hence $\vhi(X)=X$ for each $\vhi\in \inn(Q)$. But this means that $X$ is a normal subloop of $Q$. Then $X$ induces an abelian congruence of $Q$ by Proposition \ref{u4}.

Now suppose that $\nuc(Q)=1$. Then $Q$ admits triality automorphisms, cf.~Section \ref{triality}. Since $Q$ is classically solvable, it has a series of normal subloops of $Q$ with factors being commutative groups. In particular, $Q$ contains a nontrivial normal commutative subgroup $A$, the first nontrivial term of the series. Let $p$ be any prime dividing $|A|$, necessarily $p\ne 3$. By Lemma \ref{Lm:ElAb}, $Q$ contains a nontrivial normal $p$-subgroup $S$ (contained in $A$). By Proposition \ref{c2}, $\mlt(Q)$ contains a nontrivial normal commutative triality subgroup. By Proposition \ref{u6}, $Q$ contains a nontrivial normal subloop that induces an abelian congruence of $Q$.
\end{proof}


\begin{thebibliography}{99}

\bibitem{Bruck}
R.~H.~Bruck, A Survey of Binary Systems, \emph{Springer-Verlag}, 1971.

\bibitem{Doro}
S.~Doro, \emph{Simple Moufang loops}, Math. Proc. Cambridge Philos. Soc. \textbf{83} (1978), no.~\textbf{3}, 377--392.

\bibitem{DrapalMouf}
A.~Dr\'apal, \emph{A simplified proof of Moufang's theorem}, Proc. Amer. Math. Soc. \textbf{139} (2011), no. \textbf{1}, 93--98.

\bibitem{dranil}
A.~Dr\'apal, \emph{A short proof for the central nilpotency of Moufang loops of prime power order}, submitted.

\bibitem{dravojt}
A.~Dr\'apal and P.~Vojt\v{e}chovsk\'y, \emph{Abelian congruences and solvability in Moufang loops}, submitted.

\bibitem{FM}
R.~Freese and R.~McKenzie, \emph{Commutator theory for congruence modular varieties}, London Mathematical Society Lecture Note Series \textbf{125}, Cambridge University Press, Cambridge, 1987.

\bibitem{gagcyc}
S.M.~Gagola, III, \emph{Cyclic extensions of Moufang loops induced by semi-automorphisms}, J. Algebra Appl. \textbf{13} (2014), no.~\textbf{4}, 1350128, 7 pp.

\bibitem{gag3}
S.M.~Gagola, III, \emph{When are inner mapping groups generated by conjugation maps?}, Arch. Math. (Basel) \textbf{101} (2013), 207--212.

\bibitem{GlaubermanII}
G.~Glauberman, \emph{On loops of odd order. II.}, J.~Algebra \textbf{8} (1968), 393--414.

\bibitem{GW}
G.~Glauberman and C.R.B.~Wright, \emph{Nilpotence of finite Moufang $2$-loops}, J. Algebra \textbf{8} (1968), 415--417.

\bibitem{Hall}
J.I.~Hall, \emph{Moufang loops and groups with triality are essentially the same thing}, Mem. Amer. Math. Soc. \textbf{260} (2019), no.~\textbf{1252}.

\bibitem{Moufang}
R.~Moufang, \emph{Zur Struktur von Alternativk\"orpern}, Math. Ann. \textbf{110} (1935), no. \textbf{1}, 416--430.

\bibitem{Pflugfelder}
H.O.~Pflugfelder, Quasigroups and Loops: Introduction, \emph{Heldermann}, Berlin (1990).

\bibitem{Phillips}
J.D.~Phillips, \emph{Moufang loop multiplication groups with triality}, Rocky Mountain J.~Math. \textbf{29} (1999), no.~\textbf{4}, 1483--1490.

\bibitem{svcom}
D.~Stanovsk\'y and P.~Vojt\v{e}chovsk\'y, \emph{Commutator theory for loops}, J.~Algebra \textbf{399} (2014), 290--322.

\bibitem{svab}
D.~Stanovsk\'y and P.~Vojt\v{e}chovsk\'y, \emph{Abelian extensions and solvable loops}, Results Math. \textbf{66} (2014), 367--384.

\end{thebibliography}
\end{document}